\documentclass[a4paper,10pt]{article}
\usepackage[scaled]{helvet}              
\usepackage[T1]{fontenc}                 
\usepackage{geometry}        
\usepackage{exscale}      
\usepackage{a4wide}       
\usepackage{framed}       
\usepackage{subscript}    
\usepackage{authblk}      
\usepackage{appendix}     
\usepackage{colordvi}     
\usepackage{enumitem}     
\usepackage{algorithm}
\usepackage{algpseudocode} 
\usepackage[american]{babel}
\usepackage{fancyhdr}                   
\usepackage{graphicx}
\usepackage{float}    
\usepackage{array}    
\usepackage[colorlinks=true,linkcolor=black,citecolor=black,filecolor=black,urlcolor=black,hypertexnames=false]{hyperref} 
\usepackage{aliascnt}      
\usepackage{cleveref}      
\usepackage[babel,german=guillemets]{csquotes}
\usepackage[backend=bibtex,style=numeric-comp,useprefix=true,hyperref=true,firstinits=true]{biblatex}   
\bibliography{../../../Import/References/references.bib}

\usepackage[fleqn]{amsmath} 
\usepackage{mathtools}      
\usepackage{amssymb}        
\usepackage{latexsym}       
\numberwithin{equation}{section}      
\input{../../../Import/Environments/thm-environments.tex}

\crefname{figure}{Figure}{Figures}
\crefname{algorithm}{Algorithm}{Algorithms}
\crefname{subsection}{Section}{Section}
\crefname{section}{Section}{Section}
\crefname{theorem}{Theorem}{Theorems}
\crefname{lemma}{Lemma}{Lemmata}
\crefname{definition}{Definition}{Definitions}
\crefname{remark}{Remark}{Remarks}

\usepackage{xargs}   
\newcommand{\vdW}{van der Waals}
\newcommand{\Avogadro}{Avogadro}

\newcommand{\vphi}{\ensuremath{ \varphi }}

\newcommand{\setN}{\ensuremath{ \mathbb{N} }}                   
                   
\newcommand{\setR}{\ensuremath{ \mathbb{R} }}                   
                                                           
\newcommand*{\brac}[1]{\ensuremath{ \left( {#1} \right) }}        
\newcommand*{\sqbrac}[1]{\ensuremath{ \left[ {#1} \right] }}      
\newcommand*{\cbrac}[1]{\ensuremath{ \left\{ {#1} \right\} }}     
\newcommand*{\Hence}{\ensuremath{ ~~\Longrightarrow~~ }}          
\newcommand*{\Equivalent}{\ensuremath{ ~~\Longleftrightarrow~~ }} 
\renewcommand*{\vec}[1]{ \ensuremath{ \boldsymbol{#1} }}  
\newcommand*{\scp}[2]{\ensuremath{ \left( {#1}~,~{#2} \right) }}                
\newcommand*{\dualp}[2]{\ensuremath{ \left\langle {#1}~,~{#2} \right\rangle }}  
\newcommand*{\grad}{\vec{\nabla}\!}                       
\newcommand*{\dert}[1][t]{\ensuremath{ \partial_{{#1}} }} 
\newcommand*{\derr}[1][x]{\ensuremath{ \frac{d}{d{#1}} }} 
\newcommand{\Int}[4]{\ensuremath{ \int\limits_{#2}^{{#3}} {#4} d{#1} }}    
\newcommandx*{\Intdt}[3][1=0,2=T,usedefault]{ \Int{t}{{#1}}{{#2}}{{#3}} }  
\newcommandx*{\fspace}[4][1=-1pt,usedefault]{ \ensuremath{  {{#2}} \hspace{#1} \left( {#3}{#4} \right) }} 
\newcommandx*{\norm}[2]{ \ensuremath{ \left\| {#1} \right\|_{{#2}}   }}                                   
\newcommandx*{\Lp}[4][1=\Omega,2= ,usedefault]{ \fspace{ L^{{#3}}_{{#2}} }{ {#1} }{ {#4} } } 
\newcommandx*{\Wkp}[5][1=\Omega,2= ,usedefault]{ \fspace{ W^{{#3},{#4}}_{{#2}} }{ {#1} }{ {#5} } } 
\newcommandx*{\Hk}[4][1=\Omega,2= ,usedefault]{ \fspace{ H^{{#3}}_{{#2}} }{ {#1} }{ {#4} } }       
\newcommand{\concentration}{number concentration}
\newcommand{\concentrationMol}{molar concentration}
\newcommand{\massflux}{number mass flux}

\newcommand*{\vecu}{\vec{u}}
\newcommand*{\vecv}{\vec{v}}
\newcommand*{\vecj}{\vec{j}}
\newcommand*{\vecJ}{\vec{J}}
\newcommand*{\vecnu}{\vec{\nu}}
\newcommand*{\vecF}{\vec{F}}

\newcommand*{\OmegaT}[1][T]{\ensuremath{ \Omega_{#1} }} 
\newcommand{\unitConc}{ \sqbrac{m^{-3}} }               
\newcommand{\unitConcMol}{ \sqbrac{mol\,m^{-3}} }       
\newcommand{\unitFlux}{ \sqbrac{m^{-2}s^{-1}} }         
\newcommand{\unitPDE}{ \sqbrac{m^{-3}s^{-1}} }          
\newcommand{\unitPDEMol}{ \sqbrac{mol\,m^{-3}s^{-1}} }  
\newcommand*{\dl}[1][]{\ensuremath{ d_{{#1}} }}            
\newcommand{\fieldFlow}{\ensuremath{ \vecu }}              
\newcommand{\nl}[1][]{\ensuremath{ c_{*{#1}} }}            
\newcommand{\nlflux}[1][]{\ensuremath{ \vecJ_{*{#1}} }}    
\newcommand{\nlfluxrel}[1][]{\ensuremath{ \vecj_{*{#1}} }} 
\newcommand{\cl}[1][]{\ensuremath{ c_{{#1}} }}             
\newcommand{\clstart}[1][]{\ensuremath{ c_{0{#1}} }}       
\newcommand{\Rl}[1][]{\ensuremath{ R_{{#1}} }}             
\newcommand{\spaceT}{ \Hk[][0]{1}{} }                                                  
\newcommand{\spaceTh}{ V_h }                                                           
\newcommand{\spaceS}{ \fspace{L^2}{I}{; \spaceT } \cap \fspace{H^1}{I}{; \spaceT^* } } 
\newcommand{\spaceSshort}{ X }                                                         

\begin{document}
\title{Including \vdW\ Forces in Diffusion-Convection Equations -- Modeling, Analysis, and Numerical Simulations}
\author[1]{Matthias Herz} 
\author[1]{Peter Knabner}
\renewcommand\Affilfont{\itshape\small}
\affil[1]{Department of Mathematics, University Erlangen--Nuremberg, Cauerstr. 11, D-91058 Erlangen, Germany}
\date{\today}
\maketitle 
\begin{abstract}
This paper presents a model of \vdW\ forces in the framework of diffusion--convection equations. 
The model consists of a nonlinear and degenerated diffusion--convection equation, which furthermore can be considered as a model for slow perikinetic coagulation. 
For the analytical investigation, we transform the model to a porous medium equation, which provides us access to the comprehensive analytical results for porous medium equations. 
Additionally, this transformation reveals a new application for porous medium equations. 
Eventually, we present numerical simulations of the model by solving the porous medium equation. 
We note that we solve the porous medium equation without any further regularization, which is often applied in this context.\\[2.0mm]
\textbf{keywords: } \vdW\ forces, \vdW\ equation of state, coagulation, nonlinear diffusion equations, porous medium equation, numerical simulations. 
\end{abstract}
%
%
\section{Introduction}\label{sec:Introduction}
%
Since the discovery of \vdW\ forces, great efforts have been made in order to capture the physical origin of these forces. Many renowned scientists have contributed toward a better understanding, 
either by developing new ideas, or by carrying over new insights from other areas of physics, see \cite{Hunter-book,Israelachvili-book,Russel-book,Parsegian-book}. 
\par
One of the most established models of \vdW\ forces is the so-called \vdW\ equation of state. 
This equation contains a cohesion pressure, which originates from ever present attractive \vdW\ forces. 
Besides the \vdW\ equation of state, further \vdW\ force models have been invented and the present-day research on developing such \vdW\ force models 
points into the direction of investigating the quantum nature of \vdW\ forces. Thus, the spatial scales of interest have reached atomistic scales.
\medskip
\par
Although the descriptions of \vdW\ forces has become more sophisticated, researches working on continuum models for reactive transport, 
fluid flow, and elasticity have almost no access to this research at atomistic scales. The reason for this is, that in most cases no 
satisfying method for connecting atomistic models with continuum models exist. In fact, the large field of multiscale modeling tries to 
bridge the gap between different spatial scales. However, connecting in particular atomistic and continuum scales via multiscale modeling 
is still a young and emerging research field, where most of the work needs yet to be done. At the same time, the well-established continuum models 
are usually still the most powerful models for theoretical and computational investigations on macroscopic spatial scales. Nevertheless, it is not 
longer sufficient to solely investigate continuum models in combination with a heuristic description of the effective coefficients. Indeed, the physical processes 
at an atomistic scale have to be incorporated, e.g., for crack modeling and many biological systems. Consequently, one of the most important tasks 
is to find a sound strategy of how to incorporate \vdW\ forces into existing continuum equations. 
\medskip
\par
This paper exactly provides such a strategy of including \vdW\ forces. More precisely, in \cref{sec:GovEq}, we derive the model equations. 
Firstly, we show how to recover standard diffusion--convection equations and by taking \vdW\ equation of state into account, we deduce an extended diffusion--convection equation, 
which leads after a transformation to a porous medium equation. In \cref{sec:Ana}, we show global existence for the model and finally, in \cref{sec:Sim}, 
we numerically solve the resulting porous medium equation by a fixed point approach without using any regularization technique.   
%
\section{Modeling}\label{sec:GovEq}
%
\subsection{Basic concepts}
In this section, we consider a given chemical species inside a domain~$\Omega$ and we observe this single chemical species over a certain time interval~$[0,T]$. 
Here, the domain~$\Omega$ is a pure fluid domain. This means in the context of porous media, we are on the pore scale, looking inside a single pore, cf.~\cite[Chapter~1]{bear-book}.
\medskip
\par 
First of all, we briefly derive the equation that governs the kinetics on continuum scales. In contrast to atomistic scales, single particles are not longer resolved on continuum scales.  
Instead, on continuum scales, we simultaneously consider a large number of particles of a given chemical species. This approach leads to averaged kinetics, that are based on mass continuity. 
In order to formulate the mass continuity equation, we introduce commonly used notation:  
\begin{enumerate}[align=left, label=(\roman*), leftmargin=*, topsep=2.0mm, itemsep=-0.7mm]
  \item In a representative elementary volume (REV)~$V$~$[m^{-3}]$, we assume that~$N$ particles of the given single chemical species are present.
	To simultaneously track these particles, we define the \concentration, cf.~\cite[Chaper~6]{Masliyah-book}, by
        \begin{align}\label{eq:defNumberConc}
           \nl:=NV^{-1} \qquad\sim\unitConc 
        \end{align}
        as an averaged quantity. Moreover, in the following, we identify the given chemical species with its \concentration~$\nl$. 
  \item To describe the average movement of the chemical species~$\nl$, we suppose that the concentration~$\nl$ moves with the averaged velocity field $\vecv$. 
	We now define for the chemical species~$\nl$ the corresponding \massflux~$\nlflux$~$[m^{-2}s^{-1}]$, cf.~\cite[Chapter~2]{oden-book}, by $\nlflux := \nl\vecv$~$\sim[m^{-2}s^{-1}]$. 
  \item We assume the particles of the given chemical species move inside a fluid of velocity~$\fieldFlow$~$[ms^{-1}]$. 
        Hence, each and every of the particles is transported at least partially due to convection and $\nlflux$ contains a 
        convection term~$\nl\fieldFlow$~$[m^{-2}s^{-1}]$. Thus, the relative movement of the chemical species~$\nl$ with respect to the 
        fluid flow field~$\fieldFlow$ is described by the so-called drift \massflux~$\nlfluxrel:=\nlflux -\nl\fieldFlow$~$\unitFlux$.
\end{enumerate}
With the just defined quantities, we now formulate the mass continuity equation, which is given in the case of nonreactive mass transport, cf.~\cite[Chapter~2]{oden-book}, by
\begin{alignat}{2}\label{eq:massBalance-1}
\dert \nl + \grad\cdot\sqbrac{ \nl\fieldFlow + \nlfluxrel} ~=~0 \qquad     &\text{in }&\Omega \qquad \sim \unitPDE ~ .
\end{alignat}
%
\subsection{Mass balance equation}
We now assume, that a given driving force $\vecF$~$[N]$ generates the drift \massflux~$\nlfluxrel$. By multiplying this force~$\vecF$ by the \concentration~$\nl$, 
we obtain the corresponding body force density $\nl\vecF$~$[Nm^{-3}]$. 
\par
However, each chemical species has only a limited capability to react to a body force density in the sense that the magnitude of the induced particle movement is limited. 
This limitation is described by the so called mobility~$\omega$~$[s/kg=m/(Ns)]$ of the chemical species \cite[Chapter~6]{Masliyah-book}.
Hence, the induced drift \massflux\ and its generating body force density are proportional to each other and the constant of 
proportionality is given by the mobility~$\omega$, cf. \cite[Chapter 6]{Masliyah-book}. According to Einstein--Smoluchowski relation, 
cf. \cite[Chapter~6]{Masliyah-book}, we express the mobility~$\omega$ in terms of the Boltzmann constant~$k_b$~$[JK^{-1}]$, 
the temperature~$T$~$[K]$, and the diffusivity~$d$~$[m^{2}s^{-1}]$, i.e., we have $\omega ~=~ d(k_bT)^{-1}$~$[mN^{-1}s^{-1}]$~. 
Thus, in our case, the body force density $\nl\vecF$ leads to the drift mass flux~$\nlfluxrel$, which is given by
\begin{align}\label{eq:defDriftMassFlux}
 \nlfluxrel  ~=~  \omega \nl \vecF  ~=~ \frac{d}{k_bT} \nl\vecF \qquad\sim[m^{-2}s^{-1}]~.
\end{align}
Consequently, the mass continuity equation \eqref{eq:massBalance-1} now reads as
\begin{align}\label{eq:massBalance-2}
  \dert \nl + \grad\cdot\sqbrac{ \nl\fieldFlow + \frac{d}{k_bT} \nl\vecF } ~=~0 \qquad   &\text{in }\Omega \qquad\sim\unitPDE ~ .
\end{align}
%
\subsection{Modeling the drift mass flux}
We suppose that the drift movement of the particles of the \concentration~$\nl$ is induced only by the partial pressure~$p$~$[Nm^{-2}]$, 
caused by the collisions with the other particles. Here, the partial pressure~$p$ is the higher, 
the more the particles collide with each other. Hence, the particles move from regions of high pressure to regions of low pressure, 
since the frequent collision within the high pressure regions \enquote{pushes the particles away}. Whereas in the low pressure regions, the rare collision permit the
particles to stay. As a consequence, the particles move down the pressure gradient and the driving force is given, cf.~\cite[Chapter~2]{Castellanos-book}, by   
\begin{align*}
  \vecF ~:=~  -\frac{1}{\nl} \grad p  \quad \sim [N]~.
\end{align*}
Hence, we see from equation~\eqref{eq:defDriftMassFlux} that the corresponding drift \massflux~$\nlfluxrel$ is given by 
\begin{align*}
  \nlfluxrel = - \frac{d}{k_bT} \grad p \qquad[m^{-2}s^{-1}]~.
\end{align*}
Inserting this ansatz for the drift \massflux~$\nlfluxrel$into equation~\eqref{eq:massBalance-2}, leads to the mass balance equation
\begin{align}\label{eq:massBalance-3}
  \dert \nl + \grad\cdot\sqbrac{ \nl\fieldFlow - \frac{d}{k_bT} \grad p } ~=~0 \qquad   &\text{in }\Omega \qquad\sim\unitPDE ~ .
\end{align} 
%
\subsection{Classical diffusion--convection equation}\label{subsec:ClassicalModel}
Provided that the collisions between the particles of the \concentration~$\nl$ are purely elastic collisions, 
the particles solely transfer their momenta during the collisions, cf.~\cite[Chapter~1]{atkins-book}. Consequently, 
the particles do not interact during the collisions through any kind of pair interaction. Since the ideal gas law 
is based exactly on this noninteracting assumption \cite[Chapter 1]{atkins-book}, we express the pressure~$p$ by means of the ideal gas law, i.e., we obtain
\begin{align}\label{eq:idealGasLaw}
 pV = k_b N T ~~\sim[J] \quad\Equivalent\quad p = \frac{N}{V}k_b T =~ \nl k_b T ~~\sim [Nm^{-2}]~.
\end{align}
Substituting this ansatz for the pressure into equation~\eqref{eq:massBalance-3}, we arrive at the classical diffusion--convection equation
\begin{align}\label{eq:massBalance-4}
  \dert \nl + \grad\cdot\sqbrac{ \nl\fieldFlow - d \grad \nl } ~=~0 \qquad   &\text{in }\Omega \qquad\sim\unitPDE ~ .
\end{align}
In the preceding equation, the term $-d\grad\nl$ is known as the Fickian diffusion term \cite[Chapter 2]{Probstein-book}. This term models on continuum scales exactly the 
kinetics that are induced by the random collision between the particles on atomistic scales. Our derivation illustrates that Fickian diffusion is based on the assumption
of elastic collisions without any further involved pair interactions. 
%
\subsection{An diffusion--convection equation including \vdW\ interactions} 
Henceforth, we assume that the particles of the \concentration~$\nl$ interact during the collision process through \vdW\ interactions. 
Since \vdW\ interactions between particles of the same chemical species are always attractive, \vdW\ interactions may keep the 
particles together after a collision and thus the particles may stick together and build up agglomerates. 
Attractive interactions of this type are known as cohesion. Cohesion forces are included in the \vdW\ equation of state, cf. \cite[Chapter 1]{atkins-book,Parsegian-book}
\begin{align*}
 \brac{p + \frac{a}{N_A^2}\frac{N^2}{V^2}}\;\brac{V-\frac{Nb}{N_A}}= N k_b T  \qquad~\sim[J] ~.
\end{align*}
Here, $N_A$~$[mol^{-1}]$ is the \Avogadro\ constant and $b$~$[m^3mol^{-1}]$ is the average volume of a particles. The parameter $a$~$[Nm^4mol^{-2}]$ is the cohesion coefficient, 
which is in our case always positive, cf. \cite[Chapter 1]{atkins-book,Parsegian-book}. The cohesion coefficient~$a$ is a measure for the strength of the involved \vdW\ forces. 
The above \vdW\ equation of state is the crucial linking point for including \vdW\ forces in the mass continuity equation~\eqref{eq:massBalance-3}. 
\par
First, we impose the simplifying assumption, that the volume of the particle is of negligible order of magnitude, i.e., $b \approx 0$. 
Hence, we are able to simplify the \vdW\ equation of state. Second, we include \vdW\ interactions in equation~\eqref{eq:massBalance-3}, 
by expressing the pressure~$p$ by means of the simplified \vdW\ equation of state. Thereby we obtain together with equation~\eqref{eq:defNumberConc} 
\begin{align}\label{eq:vdWgaslaw}
p = k_b T \nl  - \frac{a}{N_A^2}\;\nl^2 ~~~\sim\sqbrac{Nm^{-2}} ~.     
\end{align}
We now substitute the ansatz~\eqref{eq:vdWgaslaw} for the pressure into equation~\eqref{eq:massBalance-3}. This leads us to the mass balance equation 
\begin{flalign}\label{eq:massBalance-5-nl}
&\dert \nl  + \grad \cdot \brac{ \nl\vecu - \dl\grad\nl +\frac{2a\dl}{N_A^2k_bT}\;\nl \grad\nl} ~=~ 0 \quad\text{in }\Omega  \quad\sim\unitPDE~.&
\end{flalign} 
This mass balance equation is the model equation, that includes \vdW forces into the framework of diffusion--convection equations. 
Furthermore, the preceding equation is an extension of the standard linear diffusion--convection equation, since compared to equation~\eqref{eq:massBalance-4}, we have an 
additional term in the mass flux. Exactly this additional term is the reason why the above equation is a nonlinear mass balance equation. 
\medskip
\par
However, depending on $d$, $a$, $T$, the quotient~$2da(N^2_Ak_bT)^{-1}$ can be very small, as it is $N_A^2k_b\sim 10^{23}$. 
To avoid the strong influence of the product $N_A^2k_b$, we scale the above equation by introducing the \concentrationMol~$\cl$
\begin{align}\label{eq:defMolarConc}
 \cl:=N(N_AV)^{-1} ~~\sim\unitConcMol \quad\Equivalent\quad  N_A \cl = \nl ~~\sim\unitConc~.
\end{align}
Substituting the \concentrationMol~$\cl$ into equation \eqref{eq:massBalance-5-nl} and dividing by $N_A$, yields the mass balance equation for the \concentrationMol\
\begin{flalign}\label{eq:massBalance-5}
& \dert \cl + \grad \cdot \brac{\cl\vecu - \dl\sqbrac{1 - \frac{2a}{N_Ak_bT}\;\cl} \grad\cl} = 0 \quad\text{in }\Omega  \quad\sim\unitPDEMol~.&
\end{flalign}
Here, the quotient~$2da(N_Ak_bT)^{-1}$ is only mildly influenced by the product $k_bN_A$, as $N_Ak_b\sim 10^1$.
%
\subsection{Discussion of the extended diffusion--convection equation}\label{subsec:Interpretation}
\textbf{Kinetics: } the derivation of equation~\eqref{eq:massBalance-5} leads to overall collision kinetics, which comprise of the following two competing processes 
\begin{enumerate}[itemsep=-0.5mm, topsep=0.5mm]
\item The particles of the considered chemical species~$\cl$ transfer their momenta during collisions. In case of purely elastic collisions, 
      this is the only interaction between the particles, which is in equation~\eqref{eq:massBalance-5} 
      described by the classical Fickian diffusion term $-\dl\grad\cl$. This term leads to particle spreading.
\item The particles of the chemical species~$\cl$ interact due to \vdW\ interactions during collisions. 
      In absence of any other interactions, the attractive \vdW\ interactions keep the particles together and lead immediately to aggregation. 
      We modeled this in equation~\eqref{eq:massBalance-5} by the term $-2\dl a\cl/(N_Ak_bT) \grad\cl$, which we call the cohesion term in the following. 
      This term leads to particle aggregation.
\end{enumerate}
\medskip
\par\noindent
\textbf{Connection to perikinetic coagulation: } in colloid science, the aggregation of particles is called coagulation and especially coagulation 
induced by diffusion is named perikinetic coagulation, see \cite[Chapter~1.6,12.8]{Hunter-book}. Perikinetic coagulation occurs when particles collide 
due to Brownian motion and attractive interactions keep the particles after this collisions together. The process of perikinetic coagulation can be formulated 
in terms of a reaction rate, in which only particles of the same chemical species are involved. This reaction rate~$\Rl[p]$ is commonly formulated in radially 
symmetrical situations, see \cite[Equation~(12.8.6)]{Hunter-book}. Transforming this rate function~$\Rl[p]$ to non-radially symmetrical situations, we obtain 
\begin{align*}
 \Rl[p](\cl) ~=~ -\alpha\dl\Delta(\cl^2) ~=~ -2\alpha \dl \grad\cdot(\cl \grad\cl) \qquad\sim\unitPDEMol~.
\end{align*}
The interpretation of the above reaction rate is as follows: the factor $-\dl\grad\cl$ describes the Fickian diffusion process, 
which causes the particles to collide with each other, i.e., this term brings the particles into contact. The \concentrationMol~$\cl$ measures how many 
particles are present at a given point in space. Thus, the factor $-\dl\cl\grad\cl$ describes how many particles collide with each other at this point.
Taking the divergence of $-\dl\cl\grad\cl$  yields a source or sink term $-\grad\cdot(\cl \grad\cl)$, which is a measure for mass production/destruction during the just
illustrated process at a given point in space. We take exactly this source or sink term as the reaction rate and the factor $2\alpha$ as the reaction rate constant. 
Moreover, by choosing $\alpha=a(N_Ak_bT)^{-1}$, we arrive at the diffusion--convection-coagulation model (units:\unitPDEMol)
\begin{align*}
              & \dert \cl + \grad\cdot\sqbrac{\cl\fieldFlow - \dl \grad\cl} = \Rl[p](\cl)                              & \text{in }\Omega~, \\[2.0mm]
\Equivalent~  & \dert \cl + \grad \cdot \brac{\cl\fieldFlow - \dl\sqbrac{1 - \frac{2a}{N_Ak_bT}\;\cl} \grad\cl}=0      & \text{in }\Omega~.
\end{align*}
This demonstrates, that we can interpret the mass balance equation~\eqref{eq:massBalance-5} as a diffusion--convection-coagulation model in case of perikinetic coagulation.       
\medskip
\par\noindent
\textbf{The nonlinear diffusion coefficient: } 
equation~\eqref{eq:massBalance-5} is a nonlinear diffusion--convection equation, where the nonlinearity stems from the nonlinear diffusion term
\begin{align}\label{eq:defNonlinDiffCoef}
  D(\cl) \grad\cl \qquad \text{with }\qquad  D(\cl):=\dl\sqbrac{1- \frac{2a}{N_Ak_bT}\; \cl } \qquad\sim\sqbrac{m^2 s^{-1} }~.
\end{align}
Mathematically, $D(\cl)$ is a nonlinear diffusion coefficient, which arises from the interplay between Fickian diffusion and cohesion forces. 
Since the Fickian diffusion and cohesion forces describe two competing processes, we have to distinguish the following two cases:
\begin{enumerate}[itemsep=-0.2mm, topsep=0.5mm]
  \item In case of a dominant Fickian diffusion term~$-\dl\grad\cl$, only a small amount of particles form aggregates after collisions. 
        Here, the overall kinetics is spreading and we have a nonnegative nonlinear diffusion coefficient~$D(\cl)\geq 0$.
        \hspace{1.0mm}
  \item In case of a dominant cohesion term~$2\dl a\cl/(N_Ak_bT) \grad\cl$, most of the particles form agglomerates after collisions due to strong cohesion forces. 
        Therefore, the overall kinetics is coagulation. Here, we have a negative nonlinear diffusion coefficient~$D(\cl) ~<~0$.
\end{enumerate}     
Thus, exactly in case of dominant cohesion, the mathematical model~\eqref{eq:massBalance-5} and its reformulation~\eqref{eq:massBalance-5-reformulated} become meaningless, 
as we have a negative nonlinear diffusion coefficient~$D(\cl)$.%
\footnote{We call a mathematical model meaningful in case it possesses weak solutions in the sense of \cref{def:WeakSolution}. 
Generally any (non)linear diffusion--convection equation is only meaningful in case the (non)linear diffusion coefficient~$D(\cl)$ is nonnegative, i.e., $D(\cl)\geq0$.} 
The physical reason behind this is that in case of dominant cohesion the particles coagulate to such a large extent that this is equivalent 
to a phase transition from a dissolved phase to a solid phase.
We note that the model equation~\eqref{eq:massBalance-5} is formulated as an averaged equation for the \concentrationMol~$\cl$, 
which does not account for the physical state of the particles. Nevertheless, the presented model is able to resolve phase transitions in the sense 
that the nonlinear diffusion coefficient~$D(\cl)$ becomes negative exactly in such situations.
\medskip
\par 
\noindent\textbf{The cohesion coefficient: } 
dominant cohesion forces occur in a solution only in a supersaturated situation. 
To account for this, we denote by $\cl^\ast$ the value of the concentration in case of equilibrium solubility. Thus, we have 
supersaturation in case of $\cl>\cl^\ast$ and undersaturation in case of $\cl<\cl^\ast$. Determining the cohesion coefficient~$a$ by
\begin{align*}
  a := \frac{N_Ak_bT}{2\cl^\ast}  \qquad\Equivalent\qquad \frac{2a}{N_Ak_bT} = \frac{1}{\cl^\ast} ~\\[-7.0mm]
\end{align*}
and substituting this relation in equation~\eqref{eq:massBalance-5}, leads to 
\begin{align}\label{eq:massBalance-5-reformulated}
  \dert \cl + \grad \cdot \brac{\cl\fieldFlow - \dl\sqbrac{1 - \frac{1}{\cl^\ast}\;\cl} \grad\cl} ~=~ 0 \qquad \sim \unitPDEMol~.
\end{align}
This reformulation clearly shows, that dominant cohesion occurs, if $\cl>\cl^\ast$ which is now equivalent to $D(\cl)<0$. 
Hence, we have scaled the model such that $D(\cl)<0$ solely occurs in supersaturated situations.
%
\section{Analysis}\label{sec:Ana}
%
\subsection{Weak formulation of the nonlinear model}\label{subsec:weakModel}
To present the weak formulation of the nonlinear model, we firstly introduce some notation.
\begin{enumerate}[align=left, label=({N}\arabic*), ref=({N}\arabic*), itemsep=-1.2mm]
 \item For $n\in\{1,2,3\}$, let $\Omega\in\setR^n$ be a $n$-dimensional bounded domain with boundary $\partial\Omega$ and
       corresponding exterior normal field~$\vecnu$.  
       Next, let $I:=(0,T)$ be a time interval and we introduce by $\OmegaT:= I\times\Omega$ 
       a time space cylinder with lateral boundary~$\partial\OmegaT:=I\times\partial\Omega$.  \label{NotGeom}
 \item For $1\leq p\leq\infty$, we denote the Lebesgue spaces for real valued and vector valued functions by 
       $\Lp{p}{}$, and the Sobolev spaces by $\Wkp{1}{p}{}$, cf. \cite{Adams2-book}. 
       Especially, we set $\Hk{1}{}:=\Wkp{1}{2}{}$ and $H^1_0(\Omega):=W^{1,2}_0(\Omega)$. 
       Here, the subscript $\textsubscript{0}$ denotes the functions with vanishing traces, cf. \cite{Adams2-book}. \label{NotSpace} 
 \item For a given Banach space~$V$, we refer for the definition of the Bochner spaces~$\fspace{L^p}{I}{;V}$ 
       and $\fspace{H^k}{I}{;V}$ to \cite{Evans-book} and \cite{Roubicek-book}.\label{NotBochnerSpace}
 \item We denote by $\scp{\cdot}{\cdot}_H$ the inner product on a Hilbert space~$H$ 
       and by $\dualp{\cdot}{\cdot}_{V^\ast\times V}$, the dual pairing between a 
       Banach space~$V$ and its dual space~$V^\ast$. 
       On $\setR^n$, we just write $\vecv\cdot\vecu:=\scp{\vecv}{\vecu}_{\setR^n}$.\label{NotProd}
 \item By $\spaceSshort:=\fspace{L^\infty}{I}{;\Lp{2}{}}\cap\spaceS$, we denote the solution space.  \label{NotSolutionSpace}
\end{enumerate}
We equip equation~\eqref{eq:massBalance-5}  with initial conditions and boundary conditions and define the following mathematical model:
\medskip
\par
\begin{subequations}\label{eq:strongModel}
\noindent\textbf{Mathematical model:}
\begin{align}
    \dert\cl + \grad \cdot \brac{\cl\fieldFlow -\dl\sqbrac{1 - \frac{2a \cl}{N_Ak_bT}}  \grad\cl } &= 0        &  \text{ in }  & \OmegaT,                \label{eq:strongModel-a}\\
                                                                                            \cl    &= 0        &  \text{ on }  & \partial\OmegaT,        \label{eq:strongModel-b}\\
                                                                                            \cl(0) &= \clstart &  \text{ on }  & \Omega\times\cbrac{0}.  \label{eq:strongModel-c}
\end{align} 
\end{subequations}
We multiply by a test function $\vphi\in \spaceT$, integrate by parts and arrive at the weak formulation:
\begin{definition}[Weak solution]\label{def:WeakSolution}%
 We call a function~$c\in\spaceSshort$, with $\spaceSshort$ from \ref{NotSolutionSpace},
 a weak solution of equations~\eqref{eq:strongModel-a}--\eqref{eq:strongModel-c}, iff for a.e.~$t\in I$ and $\forall~ \vphi \in \spaceT$
 \begin{flalign}\label{eq:weakModel}
   & \dualp{\dert\cl}{\vphi}_{\Hk{1}{}^\ast\times\Hk{1}{}} + \scp{\dl\sqbrac{1 - \frac{2a \cl}{N_Ak_bT}}  \grad\cl -\cl\fieldFlow }{\grad\vphi}_{\Lp{2}{}} = 0~.&
 \end{flalign}
\hfill$\square$
\end{definition}
\par
To successfully examine the above model, we introduce the following structural assumptions. 
\begin{enumerate}[align=left, label=({A}\arabic*), ref=({A}\arabic*),  itemsep=-1.0mm, topsep=-0.5mm]
 \item We assume $\dl>0$ and $a>0$ for the diffusion coefficient and cohesion coefficient. \label{AssumpCoeff}
 \item We assume for the initial data $\clstart \in \Lp{\infty}{}$ with $0\leq\clstart\leq {N_Ak_bT}(2a)^{-1}$. \label{AssumpStart}
 \item We assume for the velocity field $\fieldFlow\in\Lp[\OmegaT]{2}{}$ and that $\grad\cdot\fieldFlow = 0$ for a.e. $x\in\Omega$, a.e. $t \in [0,T]$. \label{AssumpConvection}
\end{enumerate}
Next, we prove that equation~\eqref{eq:weakModel} possesses the physical property of producing nonnegative solutions. 
\begin{lemma}\label{lem:nonneg}
Let $\cl\in\spaceSshort$ be a weak solution according to \cref{def:WeakSolution} and assume \ref{AssumpCoeff}--\ref{AssumpConvection}. Then $\cl$ is nonnegative.
\end{lemma}
\begin{proof}
We test equation~\eqref{eq:weakModel} with $\vphi=\cl[-]:=\min(\cl,0)$ and obtain for the time integral
\begin{align*}
 \dualp{\dert \cl}{\cl[-]}_{ \spaceT^\ast\times\spaceT } ~=~ \frac{1}{2} \derr[t] \norm{ \cl[-]}{\Lp{2}{}}^2~. 
\end{align*}
The convection integral vanishes with integration by parts and \ref{AssumpConvection}.
For the diffusion integral, we arrive together with \ref{AssumpCoeff} at 
\begin{align*}
     \scp{\dl\sqbrac{1 - 2a(N_Ak_bT)^{-1} c}  \grad\cl}{\grad \cl[-]}_{\Lp{2}{}} 
 ~\geq~ 0~.
\end{align*}
We now integrate in time over $(0,t)$, for some $t\in I$, and reach with \ref{AssumpStart} to
\begin{align*}
 \norm{ \cl[-](t)}{\Lp{2}{}}^2 ~\leq~ \norm{ \clstart[,-]}{\Lp{2}{}}^2 ~=~ 0 \qquad \Hence~ \cl[-](t)=0 ~\text{ for a.e }t\in I~.
\end{align*}
\end{proof}
%
\subsection{Connection to the porous medium equation}\label{subsec:pme}
In this section, we show that the derived model can be transformed to a porous medium equation. However, we note that this model is not the first aggregation model 
with the structure of a porous medium equation. More precisely, \cite{BurgerEtAl-Aggration-2009, burger_longtime_aggr, capasso_aggr} already provided a different model for aggregation with 
the structure of a porous medium equation.
\medskip
\par\noindent
\textbf{Connection to classical porous medium equations: }
Henceforth, we assume a vanishing fluid flow in equations \eqref{eq:strongModel-a}--\eqref{eq:strongModel-c}, i.e. $\fieldFlow=\vec{0}$. Next, we introduce the new variable
\begin{align}\label{eq:solutionTranformed}
\hat{\cl}:= \frac{\dl}{2}\sqbrac{1 - \frac{2a}{N_Ak_bT}\;\cl}~,
\end{align}
and we calculate the derivatives 
\begin{align*}
 \dert \hat{\cl}= - \frac{a\dl}{N_Ak_bT}\; \dert\cl \qquad\text{and} \qquad \grad\hat{\cl}=-\frac{a\dl}{N_Ak_bT}\;\grad\cl~.
\end{align*}
Furthermore, we multiply equation~\eqref{eq:strongModel-a} by $-a\dl(N_Ak_bT)^{-1}$ and thereby we obtain
\begin{flalign*}
&-\frac{a\dl}{N_Ak_bT} \; \sqbrac{\dert\cl -\grad\cdot\brac{\dl\sqbrac{1 - \frac{2a}{N_Ak_bT}\cl} \grad\cl }}=0  
 ~~\Equivalent~  
 \dert\hat{\cl}- \Delta(\hat{\cl})^2=0~.&
\end{flalign*} 
Hence, we have transformed the model \eqref{eq:strongModel-a}-\eqref{eq:strongModel-c} in case of $\fieldFlow=0$ to the following porous medium equation
\medskip
\par\noindent
\begin{subequations}\label{eq:strongModel-pme}
\textbf{Transformed mathematical model:}
\begin{align}
    \dert\hat{\cl} + \Delta(\hat{\cl})^2  &= 0                                                      &  \text{ in }  & \OmegaT,                \label{eq:strongModel-pme-a}\\
                               \hat{\cl}  &= 2^{-1}\dl                                              &  \text{ on }  & \partial\OmegaT,        \label{eq:strongModel-pme-b}\\
                            \hat{\cl}(0)  &= \frac{\dl}{2}\sqbrac{1 - \frac{2a}{N_Ak_bT}\;\clstart} &  \text{ on }  & \Omega\times\cbrac{0}.  \label{eq:strongModel-pme-c}
\end{align} 
\end{subequations}
This transformation provides us access to the comprehensive results of \cite{Vazquez-book} for porous medium equations. 
In particular, in our case of a constant nonnegative Dirichlet boundary condition~\eqref{eq:strongModel-pme-b} and nonnegative initial values~\eqref{eq:strongModel-pme-c}, 
see \ref{AssumpStart}, we know from \cite[Theorem 5.14]{Vazquez-book} that a nonnegative solution~$\hat{\cl}$ exists for all times.  
Furthermore, we obtain from the definition of $\hat{\cl}$ in equation~\eqref{eq:solutionTranformed} that 
\begin{align}\label{eq:aprioriBounded}
 \hat{\cl}\geq0 \qquad \Equivalent \qquad \cl \leq \frac{N_Ak_bT}{2a}~.
\end{align}
Thus, we immediately have existence for all times and an a~priori upper bound for the solution~$\cl$. Together with \cref{lem:nonneg}, we obtain the $L^\infty$-bound 
\begin{align*}
 \cl \in \sqbrac{0, ~\frac{N_Ak_bT}{2a} } \qquad\text{for a.e. } t\in[0,T]~.
\end{align*}
\begin{remark}\label{rem:intrinsically_stable}
This bound proves that under the given initial values and boundary data from \ref{AssumpStart} and \eqref{eq:strongModel-b}, 
for the nonlinear diffusion coefficient~$D(\cl)$ from \eqref{eq:defNonlinDiffCoef} holds intrinsically~$D(\cl)\geq0$~. 
Hence, we inherently stay in the undersaturated regime, see \cref{subsec:Interpretation}, and no complications from $D(\cl)<0$ arise. \hfill$\square$
\end{remark}
\noindent\textbf{Connection to generalized porous medium equations: }
this time we introduce the nonlinear quadratic function
\begin{align*}
\Phi:\setR\rightarrow\setR; ~~~~\cl\mapsto\Phi(\cl):=\dl\cl-\dl a(N_Ak_bT)^{-1}\cl^2~.
\end{align*}
For the function $\Phi(\cl)$, we have $\grad\Phi(\cl)={\dl}[1-2a(N_Ak_BT)^{-1}\cl]\grad\cl$~.
Substituting $\Phi(c)$ into equation~\eqref{eq:massBalance-5}, we can rewrite equation~\eqref{eq:massBalance-5} as  
\begin{align*}
  \dert\cl - \Delta\Phi(\cl) = 0~.
\end{align*}
Equations of this type are considered as generalized porous medium equations. We recover classical porous medium equations by the choice of $\Phi(\cl):=\cl^m$ for $m\geq1$.
The polynomial functions $\cl\mapsto \cl^m$ are strictly monotone increasing for all $m\geq1$, if and only if $\cl\geq 0$. 
For this reason, the assumption of $\cl\geq0$ is ensured by assuming strictly increasing functions $\Phi$. 
Exactly for strictly monotone increasing functions~$\Phi$, we again find comprehensive analytical results in \cite{Vazquez-book}. 
In our case, we have 
\begin{align*}
  \Phi^\prime(c) = \dl\sqbrac{1 - \frac{2a}{N_Ak_bT}\cl}\cl ~\geq~0 \qquad\Equivalent\qquad   0 ~\leq~ \cl ~\leq~\frac{N_Ak_bT}{2a}~. 
\end{align*}
This condition on $\cl$ holds true due to equation~\eqref{eq:aprioriBounded}. Hence, we can again access the results in \cite{Vazquez-book} by the above transformation.   
\begin{remark}\label{rem:featuresOfThePME}
 The just proved connection of the model~\eqref{eq:strongModel-a}--\eqref{eq:strongModel-c} with the porous medium equation~\eqref{eq:strongModel-pme-a}--\eqref{eq:strongModel-pme-c} 
 guarantees, that besides existence and boundedness, the presented model inherits the features of the porous medium equation. In particular, it supports a finite speed of propagation, 
 self similar solutions and in case of solutions with compact support, we have free boundaries and waiting times. For further detailed explanations concerning this features, 
 we refer again to \cite[Chapter 1]{Vazquez-book}.  \hfill$\square$
\end{remark}
%
\section{Numerical simulations}\label{sec:Sim}
In this section, we solve the transformed model~\eqref{eq:strongModel-pme-a}--\eqref{eq:strongModel-pme-c}. 
This means, we solve a porous medium equation and we obtain the solution of equations~\eqref{eq:strongModel-a}--\eqref{eq:strongModel-c} 
in a post-processing step by using the transformation~\eqref{eq:solutionTranformed}. 
\par
Among many others, the porous medium equation was solved in the past with different numerical schemes, e.g., in \cite{rose_pme,nochetto_pme,pop_pme,arbogast_pme,zhang_pme, ebmeyer_pme}. 
In \cite{rose_pme,pop_pme,nochetto_pme}, the authors used regularization techniques to solve the porous medium equation. These regularization techniques consists in 
replacing the porous medium equation
\begin{align}\label{eq:pme}
 \dert\cl -(\cl\grad\cl) = 0
\end{align}
by the regularized version
\begin{align}\label{eq:pme-reg}
 \dert\cl -([\cl+\delta]\grad\cl) = 0 \qquad \text{for some } \delta >0~.
\end{align}
The regularized equation~\eqref{eq:pme-reg} is still a nonlinear equation. However, it is not a degenerated equation, 
since the nonlinear diffusion coefficient $\cl+\delta$ is bounded from below by $\delta>0$. Whereas, in the porous medium equation~\eqref{eq:pme}, 
we have for the nonlinear diffusion coefficient $\cl$ just $\cl\geq0$. This means the degenerated case $\cl=0$ is included. 
From the regularized equation~\eqref{eq:pme-reg}, the porous medium equation~\eqref{eq:pme} is recovered in the limit $\delta\searrow 0$.
%
\subsection{Numerical scheme}
We discretize the model~\eqref{eq:strongModel-pme-a}--\eqref{eq:strongModel-pme-c} in space by a Galerkin approach and the implicit Euler~scheme in time.
Thereby, we obtain a finite dimensional but still nonlinear problem. We solve this finite dimensional nonlinear problem, by a fixed point iteration. 
\medskip
\par\noindent
To present the algorithm, we now introduce some notation:
\begin{enumerate}[align=left, label=({N}\arabic*), ref=({N}\arabic*), start=6, itemsep=-1.2mm]
 \item \textbf{Discretization in time:} we decompose the time interval $[0,T)$ into $N$~closed subintervalls~$I_n$. This means, we have $[0,T]=\bigcup_{n=1}^N I_n$. 
       By setting $I_n:=[t_{n-1},t_n]$, we obtain a sequence of time points~$(t_n)_{n=0}^N$ and for this sequence, we assume $0=:t_0 < t_1 < \ldots < t_{N-1} < t_N := T$~.
       Furthermore, suppose an equidistant sequence in the sense that the time stepping length~$\tau$ is uniformly given by $\tau:=t_{n}-t_{n-1}$ for all $n=1,\ldots N$. 
 \item \textbf{Discretization in space:} let ~$\Omega\subset\setR^{n}$, $n=2,3$. We triangulate the domain~$\Omega$ by a family of meshes~$(\mathcal{T}_h)_{h>0}$. 
       Here, the fineness of the mesh is denoted by~$h>0$.
       The elements of a mesh~$\mathcal{T}_h$ are assumed to  be quadrilaterals, which are denoted by~$K_h$. Furthermore, we suppose that the domain~$\Omega$ is a convex polygon. 
       Hence, we can triangulate the domain~$\Omega$ such that we do not make any boundary approximation errors in the triangulation procedure, i.e., we have $\mathcal{T}_h=\bigcup K_h =\Omega$. 
 \item \textbf{Discrete ansatz space:} let $x\in\Omega$ be an $n$-dimensional point, i.e., we have $x=(x_1,\ldots,x_n)^\top$. For an element~$K\in \mathcal{T}_h$, 
       we define the space $Q_1(K)$ as the space of polynomials $q(x)$ that are linear in each $x_i$, cf.~\cite[Chapter 3]{Quarteroni-book}. 
       Finally, the discrete ansatz space~$\spaceTh$ is given by 
       $\spaceTh:=\cbrac{ \vphi_h \in C^0(\Omega): ~~\vphi_h|_K \in Q_1(K) ~~\forall~ K\in T_h ~~\text{and}~~ \vphi_h|_{\partial\Omega} =0}$. 
       Note that we have $\spaceTh\subset \spaceT$.
\end{enumerate}
Next, we present the algorithm, that we used to solve the porous medium equation.
\medskip
\par\noindent
\textbf{Nonlinear continuous problem: }
for the discretization of the porous medium equation, we consider the weak formulation of equations~\eqref{eq:strongModel-pme-a}--\eqref{eq:strongModel-pme-c}, i.e., 
we look at the equation
\begin{align}\label{eq:weakModel-transformed}
  \dualp{\dert\hat{\cl}}{\vphi}_{\Hk{1}{}^\ast\times\Hk{1}{}} + \scp{ \hat{\cl}\grad\hat{\cl}}{\grad\vphi}_{\Lp{2}{}} = 0 \qquad \forall~ \vphi \in \spaceT~.
\end{align}
\textbf{Nonlinear discrete problem: }
we discretize equation~\eqref{eq:weakModel-transformed} in time with Rothe's method, cf.~\cite{rektorys-book}. This semi-discretization in time is equivalent with a Banach space valued implicit Euler scheme. 
By using Rothe's method, we obtain a sequence of solutions~$(\hat{\cl}^n)_{n=0}^N$, which are defined by the sequence of elliptic problems:
\begin{align}\label{eq:pme-semidiscrete}
   & \text{1. Set: }c^{0} := \clstart, \qquad\clstart\text{  being the initial datum of } \eqref{eq:strongModel-pme-c} \nonumber\\
   & \text{2. For } n\in\cbrac{1,\ldots,N}, \text{ solve } \forall~ \vphi \in \spaceT  \nonumber\\ 
   & \scp{\hat{\cl}^n}{\vphi}_{\Lp{2}{}} + \tau\scp{ \hat{\cl}^{n}\grad\hat{\cl}^{n} }{\grad\vphi}_{\Lp{2}{}} 
     = \scp{ \hat{\cl}^{n-1} }{\vphi}_{\Lp{2}{}} ~.
\end{align}
The sequence of stationary functions~$(\hat{\cl}^n(x))_{n=0}^N$ is supposed to converge towards the time dependent solution~$\hat{\cl}(t,x)$ 
of equation~\eqref{eq:pme-semidiscrete} with $\tau\searrow 0$, cf. \cite[Chapter 7]{Roubicek-book}. 
\par
Next, we discretize in space by writing the equation~\eqref{eq:pme-semidiscrete} over $\spaceTh$ instead of $\spaceT$. 
This means, we now search the solution in the finite dimensional space~$\spaceTh$ instead of the space~$\spaceT$. To this end, we have to project the initial values~$\clstart$ into the space~$\spaceTh$,
e.g., by the $L^2$-projection~$\Pi_h$, cf. \cite[Chapter 3]{Chen-book}. This leads us to the fully discrete problems
\begin{align}\label{eq:pme-discrete}
  & \text{1. Set  }c^{0} := \Pi_h[\clstart] \nonumber\\
  & \text{2. For } n\in\cbrac{1,\ldots,N} \text{ solve } \forall~ \vphi_h \in \spaceTh \nonumber\\ 
  & \scp{\hat{\cl}^{n}_{h}}{\vphi_h}_{\Lp{2}{}} + \tau\scp{ \hat{\cl}_{h}^{n}\grad\hat{\cl}_{h}^{n} }{\grad\vphi_h}_{\Lp{2}{}} 
    = \scp{ \hat{\cl}_{h}^{n-1} }{\vphi_h}_{\Lp{2}{}} ~.
\end{align}
Note that we obtain for each time level~$t_n$ a sequence of solutions~$(\hat{\cl}^{n}_h)_{h>0}$ and this sequence of solutions is assumed to converge with $h\searrow0$ toward 
the $n$th solution~$\hat{\cl}^n$ of equation~\eqref{eq:pme-semidiscrete}.  
\medskip
\par\noindent
\textbf{Linear discrete problem: }
equations~\eqref{eq:pme-discrete} leads to a finite dimensional but still nonlinear equation systems. However, nonlinear equation systems are not directly solvable on computers 
and we have to iteratively solve these equation systems by a fixed point method. More precisely, we start a fixed point iteration at each time level~$n$. 
In the following, the iteration index~$k\in\setN$ denotes the current iteration step of the fixed point iteration, which consists in solving the sequence of linear problems:
\begin{align}\label{eq:pme-discrete-linear}
  & \text{1. Set  }c^{0} := \Pi_h[\clstart]~. \nonumber\\
  & \text{2.1 For } n\in\cbrac{1,\ldots,N}, \text{ set } c^{n,0}_h := c^{n-1}_h~. \nonumber\\[2.0mm] 
  & \text{2.2 For } k\in\setN, \text{ solve } \forall~ \vphi \in \spaceTh \nonumber\\ 
  & \scp{\hat{\cl}^{n,k}_{h}}{\vphi_h}_{\Lp{2}{}} + \tau\scp{ \hat{\cl}_{h}^{n,k-1}\grad\hat{\cl}_{h}^{n,k} }{\grad\vphi_h}_{\Lp{2}{}} 
     = \scp{ \hat{\cl}_{h}^{n-1} }{\vphi_h}_{\Lp{2}{}} ~.
\end{align}
In \cref{algo:pme}, we schematically present the resulting algorithm. 
\begin{algorithm}[ht]
\caption{ Solving the porous medium equation}\label{algo:pme}
\begin{algorithmic}[1]
\State set~~$time\_step\_size$~~and~~$end\_time$
\State set~~$iter\_step\_max$~~and~~$tol$
\State initialize vector~~$solution~\gets~$ initial values
\State initialize vector~~$solution\_old~\gets~solution$
\State initialize vector~~$solution\_iter~\gets~solution\_old$
\State set~~$time\_step=1$
\While{ $time\_step*time\_step\_size < end\_time$ } \Comment{time stepping loop}
  \State set~~$error=\infty$  
  \State set~~$iter\_step =1$ 
  \While{ $iter\_step \leq iter\_step\_max ~~\text{\textbf{and}} ~~error \geq tol$ } \Comment{fixed point iteration loop}
    \State compute~~$solution$ by solving equation~\eqref{eq:pme-discrete-linear}
    \State compute~~$error = \|solution -solution\_iter\|_2$
    \State update~~$solution\_iter ~\gets~solution$  
    \State increment~~$iter\_step ~\gets~ iter\_step + 1$ 
  \EndWhile
  \State update~~$solution\_old ~\gets~solution$  
  \State increment~~$time\_step ~\gets~ time\_step + 1$ 
\EndWhile 
\end{algorithmic}
\end{algorithm}
This fixed point iteration scheme returns for each $h$ and for each time level~$n$ a sequence of solutions $(\hat{\cl}^{n,k}_h)_{k\in\setN}$ 
and this sequence is assumed to converge with $k\rightarrow\infty$ toward the $n$th solution~$\hat{\cl}^{n}_{h}$ of equation~\eqref{eq:pme-discrete}. 
Formally, we recover equation~\eqref{eq:pme-discrete} from equation~\eqref{eq:pme-discrete-linear} by the observation that in the limit~$k\rightarrow\infty$, 
we have $\hat{c}^{n,k-1}_h=\hat{c}^{n,k}_h=\hat{c}^{n,\infty}_h$ and by setting $\hat{\cl}^n_h:=c^{n,\infty}_h$. 
%
\subsection{Implementation}\label{subsec:impl}
The implementations have been carried out within the deal.II library, cf. \cite{deal.II}. In our 2d-computations%
\footnote{We note that due to the dimension independent way of programming in deal.II, our implementations work also for 3d-computations. 
For ease of presentation, we just show the results of the 2d-dynamics, since the qualitative behavior of the 3d-dynamics are the same as for the 2d-dynamics}, 
we used the domain~$\Omega:=[0,1]\times[0,2]$ and we chose in the original model equations~\eqref{eq:strongModel-a}--\eqref{eq:strongModel-c} 
the parameters~$\dl=1$ and $a=2^{-1}N_Ak_bT$. Thus, we investigated the model 
\begin{subequations}\label{eq:model-computed}
\begin{align}
    \dert\cl + \grad\cdot ([1-\cl] \grad\cl)^2  &= 0                                                     &  \qquad\text{ in }  & \OmegaT,                \label{eq:model_computed-a}\\
                                          \cl   &= 0                                                     &  \qquad\text{ on }  & \partial\OmegaT,        \label{eq:model_computed-b}\\
                                        \cl(0)  &= \clstart \text{ from } \eqref{eq:model_init_values}   &  \qquad\text{ on }  & \Omega\times\cbrac{0}.  \label{eq:model_computed-c}
\end{align} 
\end{subequations}
We supplemented this equations with the rough initial values, see \cref{fig:initial_solution},
\begin{align}\label{eq:model_init_values}
 \clstart(x_1,x_2):= \begin{cases}
                         1 & \text{ for }~ (0.25 \leq x_1 \leq 0.75) ~~\wedge~~ (0.5 \leq x_2 \leq 1.5)~, \\
                         0 & \text{ else }~.
                     \end{cases} 
\end{align}
However, as already mentioned before, we did not solved the above model equations. Instead, we used the simplified transformation~\eqref{eq:solutionTranformed}
\begin{align}\label{eq:solutionTranformed-simple}
 \hat{\cl} = 1-\cl~.
\end{align}
Hereby, we obtained the porous medium equation
\begin{subequations}\label{eq:pme-computed}
\begin{align}
    \dert\hat{\cl} + \grad\cdot (\hat{\cl}\grad\hat{\cl})    &= 0                                                      &  \qquad\text{ in }  & \OmegaT,                \label{eq:pme_computed-a}\\
                                                  \hat{\cl}  &= 1                                                      &  \qquad\text{ on }  & \partial\OmegaT,        \label{eq:pme_computed-b}\\
                                               \hat{\cl}(0)  &= \hat{\cl}_0 \text{ from } \eqref{eq:pme_init_values}   &  \qquad\text{ on }  & \Omega\times\cbrac{0}.  \label{eq:pme_computed-c}
\end{align} 
\end{subequations}
By using the transformation~\eqref{eq:solutionTranformed-simple}, we arrived at the initial values, see \cref{fig:initial_solution}, 
\begin{align}\label{eq:pme_init_values}
 \hat{\cl}_0(x_1,x_2):= \begin{cases}
                         0 & \text{ for }~ (0.25 \leq x_1 \leq 0.75) ~~\wedge~~ (0.5 \leq x_2 \leq 1.5)~, \\
                         1 & \text{ else }~.
                       \end{cases} 
\end{align} 
In the computations, we solved the equations~\eqref{eq:pme_computed-a}--\eqref{eq:pme_computed-c} for $\cl$ with \cref{algo:pme} and by the transformation~\eqref{eq:solutionTranformed-simple}, 
we obtained the solution~$\cl$ of equations~\eqref{eq:model_computed-a}--\eqref{eq:model_computed-c}.
\begin{figure}[ht]
  \begin{center}
    \includegraphics*[bb=30 100 1300 580, scale=0.25]{./../pictures/initial_solution.eps}
  \end{center}
  \vspace{-4mm}
 \caption{Graph of the initial values $\clstart$ from~\eqref{eq:model_init_values} (left) and $\hat{\cl}_0$ from~\eqref{eq:pme_init_values} (right)}\label{fig:initial_solution}
\end{figure} 
The above setting for our numerical simulations was motivated from the observation, that in case of boundary condition~\eqref{eq:strongModel-b} and $\fieldFlow=\vec{0}$ for both, 
the standard diffusion equation~\eqref{eq:massBalance-4} and the extended model~\eqref{eq:strongModel-a}, the solutions reach in the long time limit a homogeneous state at 
the given constant boundary value.%
\footnote{For the extended model~\eqref{eq:massBalance-5}, this is true in case where the nonlinear diffusion coefficient~$D(\cl)$ from equation~\eqref{eq:defNonlinDiffCoef} remains nonnegative, %
          see \cref{subsec:pme} and \cref{rem:intrinsically_stable}.}
However, we showed in \cref{subsec:ClassicalModel} that the standard diffusion equations~\eqref{eq:massBalance-4} is based on the assumption of noninteracting particles. 
In contrast, we included attractive \vdW\ interactions in the extended equation~\eqref{eq:massBalance-5}. 
These attractive \vdW\ interactions keep the particles of the concentration~$\cl$ together and thus slow down the propagation speed compared to the noninteracting model. 
This difference in propagation speed can be analytically validated, since for the standard diffusion equation~\eqref{eq:massBalance-4} it is known, 
that information propagates with infinite speed, cf.~\cite[Chapter~2]{Evans-book}. Whereas one characteristic feature of the porous medium equation is 
the finite propagation speed, see \cref{rem:featuresOfThePME}. In summary, the attractive \vdW\ forces become visible in the model~\eqref{eq:massBalance-5} in a change of propagation speed. 
\par
In order to validate this difference in propagation speed numerically, we chose for $\cl$ the homogeneous boundary condition~\eqref{eq:model_computed-b} 
and the initial value~$\clstart$ from equation~\eqref{eq:model_init_values}. Thus, the dynamics in our simulations were solely induced by the
initial solution profile. As already mentioned above, we computed the solution~$\cl$ of~\eqref{eq:model_computed-a}--\eqref{eq:model_computed-c} 
by solving the equations~\eqref{eq:pme_computed-a}--\eqref{eq:pme_computed-c} and using the transformation~\eqref{eq:solutionTranformed-simple}. 
We compared this solution with the solution of the standard diffusion equation~(set $\fieldFlow=\vec{0}$ and $d=1$ in equation~\eqref{eq:massBalance-4} )  
\begin{align}\label{eq:heat}
 \dert\tilde{\cl} -\Delta\tilde{\cl} = 0
\end{align}
to identical initial values and boundary values~\eqref{eq:model_computed-b}, \eqref{eq:model_computed-c}.  Furthermore, to guarantee comparability, 
we solved~$\tilde{\cl}$ exactly in the same way, as we did for~$\cl$. This means, we used the transformation~\eqref{eq:solutionTranformed-simple} 
and solved the transformed equation for $(1-\tilde{\cl})$ to initial values and boundary values~\eqref{eq:pme_computed-b}, \eqref{eq:pme_computed-c}. 
However, the transformation~\eqref{eq:solutionTranformed-simple} was the identity in this case, i.e., 
\begin{align*}
 \dert(1-\tilde{\cl}) -\Delta(1-\tilde{\cl}) = 0 \qquad\Equivalent\qquad \dert\tilde{\cl} -\Delta\tilde{\cl} = 0~.
\end{align*}
For a given mesh size~$h=2^{-7}$, we ran our simulations over~$600$~time steps. Thereby, we used the time steps size~$\tau:=10^{-4}$. 
The reason for this small time step size~$\tau$ was to ensure convergence of the fixed point method.
More precisely, to guarantee convergence, the fixed point iteration must satisfy the so-called contraction property, 
cf.~\cite[Chapter~9.2]{Evans-book}, and for parabolic equations the contraction property holds solely for sufficiently small 
time step sizes~$\tau$, cf. \cite[Chapter~9.2, Theorem~2]{Evans-book}. Furthermore, the condition number of the elliptic part of 
the system matrix associated with equation~\eqref{eq:pme-discrete-linear} increases with~$O(h^{-2})$, cf.~\cite[Chapter~3]{Knabner-FE-book}. 
Hence, the smaller the mesh size~$h$, the worser gets the condition number. However, equation~\eqref{eq:pme-discrete-linear} shows that at the same moment, 
the condition number decreases with~$O(\tau)$. Consequently, we chose a small time step size to guarantee both, 
the contraction property and a reasonable condition number. 
\par
Next, we set in \cref{algo:pme} the maximal number of iteration steps to~$40$ and the tolerance to~$10^{-8}$. 
Finally, we solved the resulting linear equation systems with the build-in version of the sparse direct solver UMFPACK, cf. \cite{umfpack}. 
\par
As already mentioned before, we did not need any regularization in \cref{algo:pme}. Hence, we were able to investigate the true behavior of the solution~$\cl$.
Finally, we clearly observed the desired difference in propagation speed of the solutions~$\cl$ of~\eqref{eq:pme_computed-a}--\eqref{eq:pme_computed-c} and $\tilde{\cl}$
of~\eqref{eq:heat}, \eqref{eq:pme_computed-b},\eqref{eq:pme_computed-c} in our simulations, 
see \cref{fig:solution_200,fig:solution_400,fig:solution_600} below. This validated numerically the attractive \vdW\ forces in the model equation~\eqref{eq:strongModel-a}.   
\begin{figure}[H]
  \begin{center}
    \includegraphics*[bb=150 40 1200 530, scale=0.25]{./../pictures/solution_pme_heat_200.eps}
  \end{center}
  \vspace{-4mm}
 \caption{Graph of $\cl$ (left) and $\tilde{\cl}$ (right) after $200$ time steps}\label{fig:solution_200}
\end{figure} 
\begin{figure}[H]
  \vspace{-4.0mm}
  \begin{center}
    \includegraphics*[bb=150 40 1200 500, scale=0.25]{./../pictures/solution_pme_heat_400.eps}
  \end{center}
  \vspace{-4mm}
 \caption{Graph of $\cl$ (left) and $\tilde{\cl}$ (right) after $400$ time steps}\label{fig:solution_400}
\end{figure}    
\begin{figure}[H]
  \begin{center}
    \includegraphics*[bb=150 40 1200 500, scale=0.25]{./../pictures/solution_pme_heat_600.eps}
  \end{center}
  \vspace{-4mm}
 \caption{Graph of $\cl$ (left) and $\tilde{\cl}$ (right) after $600$ time steps}\label{fig:solution_600}
\end{figure}   
\medskip
\par
In a second numerical experiment, we investigated the case of dominant cohesion forces. In \cref{subsec:Interpretation}, 
we pointed out that this happens in case of a negative nonlinear diffusion coefficient~$D(\cl)$, with $D(\cl)$ from equation~\eqref{eq:defNonlinDiffCoef}. 
Although the mathematical model does not longer possess a weak solution in the sense of \cref{def:WeakSolution}, 
we were able to compute numerically reasonable aggregation kinetics. To this end, we used the same computational domain as before and we solved the 
model equations~\eqref{eq:model_computed-a}--\eqref{eq:model_computed-c} together with the initial condition, see \cref{fig:solution_aggl}
\begin{align}\label{eq:model_init_values_aggl}
 \clstart(x_1,x_2):= \begin{cases}
                       1.5 & \text{ for }~ (0.4 \leq x_1 \leq 0.6)~\wedge~(0.75 \leq x_2 \leq 1.0) \\
                         1 & \text{ for }~ (0.25 \leq x_1 \leq 0.4)~\wedge~(0.6 \leq x_1 \leq 0.75)~  \\
                           & ~~~~~~~~~~       \wedge~(0.5 \leq x_2 \leq 0.75)~\wedge~( 1.0\leq x_2 \leq1.5)~, \\
                         0 & \text{ else }~.
                     \end{cases} 
\end{align}
We used the same mesh size~$h=2^{-7}$, but we chose the time step size~$\tau=10^{-6.5}$ in order to account for the negative eigenvalues. 
Our simulations ran stable over~$150$~time steps, before oscillations occurred. We expected these oscillations since aggregation leads to 
mass clustering and thus to a blow-up in the computed solution. Within the first $150$~time steps, the computations illustrated this blow-up kinetics 
and furthermore revealed that the aggregation kinetics completely froze the propagation of the support~$S(t):=\cbrac{x\in\Omega:~~\cl(t,x)>0}$ 
of the solution, see~\cref{fig:solution_aggl} below.   
\begin{figure}[H]
  \begin{center}
    \includegraphics*[bb=05 100 1500 870, scale=0.20]{./../pictures/solution_aggl_merge.eps}
  \end{center}
  \vspace{-4mm}
 \caption{Graph of the solution~$\cl$ of~equations~\eqref{eq:model_computed-a},~\eqref{eq:model_computed-b},~\eqref{eq:model_init_values_aggl} 
          after $0$ time steps (left) and after $100$ time steps (right)}\label{fig:solution_aggl}
\end{figure} 
%
\section{Conclusion}
We deduced the nonlinear, degenerated diffusion--convection equation~\eqref{eq:massBalance-5}, by substituting the pressure according to 
\vdW\ equation of state. Thereby, we incorporated cohesion forces, that originated from ever present \vdW\ forces.  
We demonstrated that the resulting nonlinear and degenerated diffusion equation turned out to be a suitable model for slow perikinetic coagulation, 
see~\cref{subsec:Interpretation}, and we were able to transform the model to a porous medium equation, see \cref{subsec:pme}. 
Thereby we illustrated a new application for porous medium equations. 
Next, we revealed that the presented model possess a weak solution in the sense of \cref{def:WeakSolution} solely in case of a nonlinear diffusion coefficient~$D(\cl)\geq0$, 
where $D(\cl)$ is given by equation~\eqref{eq:defNonlinDiffCoef}. In \cref{subsec:pme}, we showed that this is guaranteed, in case the model is given by 
equations~\eqref{eq:strongModel-a}--\eqref{eq:strongModel-c} and assumptions~\ref{AssumpCoeff}, \ref{AssumpStart} hold. 
\par
Furthermore, we used \cref{algo:pme} for solving the resulting porous medium equation without any further regularization. 
Thus, the computed solutions reflect the true degenerate character of the model equations~\eqref{eq:strongModel-a}--\eqref{eq:strongModel-c}.
Finally, in our numerical simulations we were able to validate the attractive \vdW\ forces in the model in both cases, 
the dominant cohesion case and the dominant Fickian diffusion case, see \cref{subsec:Interpretation,subsec:impl} for details.
\medskip
\par\noindent
%
\textbf{Acknowledgements}~~
We would like to thank N.~Ray and F.~Frank for carefully reading early versions of this manuscript and for giving constructive comments 
which helped us to improve the quality. 
\par
M. Herz is supported by the Elite Network of Bavaria. 
%
\addcontentsline{toc}{section}{References}  
\printbibliography
\end{document}